\documentclass[12pt]{amsart}

\usepackage{graphicx}
\usepackage[utf8]{inputenc}
\usepackage[english]{babel}
\usepackage{csquotes}
\usepackage{amssymb}
\usepackage{tikz-cd}

\usepackage{enumerate}
\usepackage{enumitem}
\usepackage{amsmath,amsfonts,pictex,graphicx,fullpage,etex,tikz,hyperref}
\newtheorem{theorem}{Theorem}[section]
\newtheorem{lemma}[theorem]{Lemma}
\newtheorem{proposition}[theorem]{Proposition}
\newtheorem{corollary}[theorem]{Corollary}

\theoremstyle{definition}
\newtheorem{definition}[theorem]{Definition}

\newtheorem{problem}[theorem]{Problem}

\theoremstyle{remark}
\newtheorem{remark}[theorem]{Remark}
\newtheorem{notation}[theorem]{Notation}
\newtheorem{claim}[theorem]{Claim}

\numberwithin{equation}{section}

\newcommand{\abs}[1]{\lvert#1\rvert}

\newcommand{\R}{\mathbb{R}}

\newcommand{\I}{\mathrm{I}}
\newcommand{\N}{\mathbb{N}}
\newcommand{\SL}{\mathrm{SL}}

\newcommand{\dd}{\mathrm{d}}
\newcommand{\SO}{\mathrm{SO}}

\newcommand{\vect}[1]{\boldsymbol{\mathbf{#1}}}
\newcommand{\inv}{^{-1}}

\newcommand{\Ad}{\mathrm{Ad}}
\newcommand{\norm}[1]{\| #1 \|}

\begin{document}

\title{Equidistribution of expanding translates of curves in homogeneous spaces with the action of $(\SO(n,1))^k$}


\author{Lei Yang}
\address{Einstein Institute of Mathematics, The Hebrew University of Jerusalem, Givat Ram., Jerusalem, 9190401, Israel.}
\email{yang.lei@mail.huji.ac.il}

\thanks{The author is supported in part by ISF grant 2095/15 and ERC grant AdG 267259.}

\date{}



\begin{abstract}
Given a homogeneous space $X = G/\Gamma$ with $G$ containing the group $H = (\SO(n,1))^k$. Let $x\in X$ such that $Hx$ is dense in $X$. Given an analytic curve 
$\phi: I=[a,b] \rightarrow H$, we will show that if $\phi$ satisfies certain geometric condition, then for a typical diagonal subgroup $A =\{a(t): t \in \R\} \subset H$ the translates $\{a(t)\phi(I)x: t >0\}$ of the curve $\phi(I)x$ will tend to be equidistributed in $X$ as $t \rightarrow +\infty$. The proof is based on the study of linear representations of $\SO(n,1)$ and $H$.
\end{abstract}

\maketitle

\section{Introduction}


\subsection{Equidistribution of expanding translates of curves in homogeneous spaces} 
\label{subsec-equidistribution-expanding-translates-curves-homogeneous}
The study of limit distributions of expanding translates of curves in homogeneous spaces is initiated by Nimish Shah 
\cite{Shah_1}, \cite{Shah_3}, \cite{Shah_2}, and \cite{Shah2010}. It is a very important topic in homogeneous 
dynamics and has interesting applications to number theory and geometry. The basic setup of this type of 
problems can be summarized as follows:
\begin{problem}
\label{problem:expanding-curves-equidistribution}
Let $H$ be a semisimple Lie group and $G$ be a Lie group containing $H$. Let $\Gamma \subset G$ be a lattice in $G$. Then the homogeneous space $X := G/\Gamma$ admits a $G$-invariant probability measure $\mu_G$. Let $A = \{a(t): t \in \R\} \subset H$ be a one-parameter diagonalizable subgroup of $H$, and let $\phi : I =[a,b] \rightarrow H$ be a compact piece of analytic curve in $H$ such that $\phi(I)$ is expanded by the conjugate action of $\{a(t): t >0\}$. One can ask what condition on $\phi$ we need to ensure that for any $x =g\Gamma \in X$ such that $Hx$ is dense in $X$, the expanding translates of $\phi(I)x$ by $\{a(t): t >0\}$, namely, $\{ a(t)\phi(I)x : t >0 \}$, tend to be equidistributed in $X$ with respect to $\mu_G$ as $t \rightarrow +\infty$. 
\end{problem}
\begin{remark}
\label{rmk:orbit-closure}
Without loss of generality, throughout this paper, we always assume that $Hx$ is dense in $X$. In fact, if $Hx$ is not dense, then by Ratner's orbit closure theorem (see \cite[Theorem A and Theorem B]{ratner_2}), its closure is $Fx$ for some analytic subgroup $F \subset G$ such that $F \cap g\Gamma g\inv$ is a lattice in $F$, then we may replace $G$ by $F$, $\Gamma$ by $F\cap g \Gamma g^{-1}$.
\end{remark}
\subsection{Main result}
\label{subsec-main-result}
\par In this paper, we consider the following special case. Let $H = (\SO(n,1))^k$ and $G$ be a Lie group containing $H$. Let us fix a maximal one parameter $\R$-split Cartan subgroup $B =\{b(t): t\in \R\}$ in $\SO(n,1)$. Given $k$ positive numbers $\eta_1, \eta_2, \dots, \eta_k$, let us define the diagonal subgroup
 \[A : = \{ a(t):= (b(\eta_1 t), b(\eta_2 t), \dots, b(\eta_k t)) \in H : t \in \R\}.\]
 \begin{remark}
 The definition of $A$ depends on $\eta_1, \eta_2, \dots, \eta_k$. However, since we fix these parameters throughout this paper, we may denote the subgroup by $A$ instead of $A(\eta_1, \eta_2, \dots, \eta_k)$ for simplicity.
 \end{remark}
 Let $\phi: I= [a,b] \rightarrow H$ be an analytic curve in $H$, and $x = g\Gamma \in X$ be such that $Hx$ is dense in $X$. We will find a geometric condition on $\phi$ such that $a(t) \phi(I) x$ tends to be equidistributed in $X$ with respect to $\mu_G$ as $ t \rightarrow +\infty$. By a standard reduction argument (see \cite[Proof of Theorem 1.2]{Shah_2} and \cite[Proof of Theorem 2.1]{Yang_2}), one can reduce the problem to the case where the curve is contained in the expanding horospherical subgroup of $H$ with respect to $A$:
 $$U_H^+ (A) : = \{ h \in H : a(-t) h a(t) \rightarrow e \text{ as } t \rightarrow +\infty\}.$$
 It is easy to see that the expanding horospherical subgroup $U^{+}(B)$ of $\SO(n,1)$ with respect to $B$ is a unipotent subgroup isomorphic to $\R^{n-1}$ with the following isomorphism: $u: \R^{n-1} \rightarrow U^{+}(B)$. This implies that 
 $$U_H^+(A) = \{(u(\vect{x}_1), u(\vect{x}_2), \dots, u(\vect{x}_k)) : \vect{x}_i \in \R^{n-1} \text{ for } i=1,2,\dots, k\}.$$
 \begin{definition}
 \label{def:u_k}
 We define 
 $$u_k :(\R^{n-1})^k \rightarrow U_H^+(A) $$
 by $u_k(\vect{x}_1, \vect{x}_2, \dots, \vect{x}_k) := (u(\vect{x}_1),u(\vect{x}_2), \dots, u(\vect{x}_k))$. 
 Then $u_k$ gives an isomorphism between $(\R^{n-1})^k$ and $U_H^+ (A)$.
  \end{definition}

 \begin{definition}
 \label{def:hyperboloids}
 Let $\|\cdot\|_2$ denote the standard Euclidean norm on $\R^{n-1}$, namely,
  $$\|\vect{x}\|_2 := (x_1^2 + x_2^2 + \cdots + x_{n-1}^2)^{1/2}$$
  for $\vect{x} =(x_1, x_2, \dots, x_{n-1}) \in \R^{n-1}$. For $\vect{x} \in \R^{n-1}\setminus\{\vect{0}\}$, let us define 
  $\vect{x}\inv : = \frac{\vect{x}}{\|\vect{x}\|_2^2}$. For $(\vect{x}_1, \vect{x}_2, \dots, \vect{x}_k) \in (\R^{n-1})^k$ with $\vect{x}_i \neq \vect{0}$ for $i=1,2,\dots, k$, let us define 
  $$(\vect{x}_1, \vect{x}_2, \dots, \vect{x}_k) \inv := (\vect{x}_1\inv, \vect{x}_2 \inv , \dots, \vect{x}_k \inv).$$ 
 \end{definition}
 
 \par The main result of this paper is the following:
 \begin{theorem}
 \label{thm:main-result}
 Let $H = (\SO(n,1))^k$, $A =\{a(t): t \in \R\} \subset H$, $G$, $\Gamma \subset G$, $X = G/\Gamma$, $\mu_G$, $x = g \Gamma \in X$ and $u_k: (\R^{n-1})^k \rightarrow U_H^+(A)$ be as above. For any analytic curve $\varphi: I = [a,b] \rightarrow (\R^{n-1})^k$
 such that for some $s_0 \in I$, $(\varphi(s) - \varphi(s_0))\inv$ is defined for almost every $s\in I$ and $\{(\varphi(s_1) - \varphi(s_0))\inv - (\varphi(s_2) - \varphi(s_0))\inv : s_1, s_2 \in I\}$ is not contained in any proper subspace of $(\R^{n-1})^k$, we have that the expanding translates $a(t)u_k(\varphi(I))x$ tend to be equidistributed in $X$ with respect to $\mu_G$ as $t \rightarrow +\infty$, namely, for any compactly supported continuous function $f \in C_c(X)$,
 $$\lim_{t \rightarrow +\infty} \frac{1}{|I|} \int_{a}^b f(a(t)u_k(\varphi(s))x) \dd s = \int_{X} f \dd \mu_G.$$
 \end{theorem}
 \begin{remark}
 \label{rmk:sl2}
 Since $\SO(2,1)$ is isomorphic to $\SL(2,\R)/\pm \I_2$, the above theorem holds for $H = (\SL(2,\R))^k$.
 \end{remark}

\subsection{Related results}
\label{subsec-related-results}
\par In \cite{Shah_1} and \cite{Shah_3}, the case where $G = \SO(m,1)$, $H = \SO(n,1)$ ($m \geq n \geq 2$) and $A =\{a(t): t \in \R\}$ is a maximal $\R$-split Cartan subgroup of $H$ is studied. It is proved that if $\phi : I \to H$ satisfies a geometric condition, then the equidistribution result holds. This result is generalized by the author \cite{Yang_2} to the case where $H = \SO(n,1)$ and $G$ is any Lie group containing $H$. In \cite{Shah_2} and \cite{Shah2010} the case where $H = \SL(n+1, \R)$ and $G$ is any Lie group containing $H$ is studied. In \cite{Shah_2}, the equidistribution result is established for the following singular diagonal subgroup:
\[A = \left\{a(t) := \begin{bmatrix} e^{nt} &  \\ & e^{-t} \I_n \end{bmatrix} : t \in \R \right\}.\]
In \cite{Shah2010}, the equidistribution result is established for more general diagonal subgroups. Both \cite{Shah_2} and \cite{Shah2010} give interesting applications to Diophantine approximation. Another special case in this direction is later studied by the author \cite{Yang_1} where the equidistribution is established when $H = \SL(2n, \R)$, $G$ is any Lie group containing $H$, and
\[A = \left\{a(t):= \begin{bmatrix} e^t \I_n & \\ & e^{-t} \I_n \end{bmatrix} : t \in \R\right\}.\]
Recently, Shah and Yang \cite{shah_yang} establish the equidistribution result for $H = \SL(m+n,\R)$,
 $$A = \left\{ a(t):= \begin{bmatrix} e^{nt} \I_m & \\ & e^{-mt} \I_n \end{bmatrix} : t \in \R \right\},$$
 and general Lie group $G$. 
Similar to \cite{Shah_2} and \cite{Shah2010}, \cite{Yang_1} and \cite{shah_yang} also have applications to Diophantine approximation. We refer the reader to \cite{Dani}, \cite{Klein_Mar} and \cite{Klein_Weiss} for the connection between Diophantine approximation and homogeneous dynamics.
\par Throughout this paper, we fix two positive integers $n\geq 2$ and $k \geq 2$, and a maximal $\R$-split Cartan subgroup $B = \{b(t) : t \in \R\}$ in $\SO(n,1)$.


\subsection{Outline of the proof}
\label{subsection:outline-of-the-proof} We will give the outline of the proof of Theorem \ref{thm:main-result} and explain the main difficulties. 
\par For $t>0$, let $\mu_t$ denote the normalized parametric measure on the curve $a(t)u_k(\varphi(I))x$, namely, for $f \in C_c (X)$, 
$$\int f \dd \mu_t : = \frac{1}{b-a}\int_a^b f(a(t)u_k(\varphi(s))x) \dd s . $$
Our goal is to show that $\mu_t \rightarrow \mu_G$ as $t \rightarrow +\infty$. Following some argument developed by Shah \cite{Shah_1} and \cite{Shah_2}, we modify $\mu_t$ to another probability measure $\lambda_t$, and show that $\lambda_t \rightarrow \mu_G$ implies $\mu_t \rightarrow \mu$ (as $t \rightarrow +\infty$). Then we prove that any accumulation point $\lambda_{\infty}$ of $\{\lambda_t : t >0\}$ is a probability measure on $X$ invariant under the action of a unipotent subgroup $W$ of $H$. To show that $\lambda_{\infty}$ is still a probability measure on $X$, one needs to apply a result of Kleinbock and Margulis \cite{Klein_Mar} and the linearization technique. Then we can apply the Ratner's theorem and the linearization technique to show that if $\lambda_{\infty} \neq \mu_G$, then for a particular linear representation $V$ of $H$, there exists a nonzero vector $v \in V$ such that for all $s \in I$, $a(t)u_k(\varphi(s))v$ converges as $t \rightarrow +\infty$. Since $A$ is a diagonal subgroup, one can decompose $V$ as the direct sum of eigenspaces of $A$:
 $$ V = \bigoplus_{\omega} V^{\omega}(A),$$
 where $V^{\omega}(A) = \{v \in V: a(t)v = e^{\omega t} v\}$. Let us define $V^{-0}(A) := \bigoplus_{\omega \leq 0} V^{\omega}(A) $, then the last conclusion in the paragraph above is equivalent to 
 \begin{equation}
 \label{eq:linear-dynamical-condition}
 \{u_k(\varphi(s))v\}_{s \in I} \subset V^{-0}(A).
 \end{equation} 
 Until now, the argument is standard and more or less the same as that in \cite{Shah_1}, \cite{Shah_2}, \cite{Shah2010}, \cite{Yang_2} and \cite{Yang_1}.
 \par Our main task is to show that if $\varphi$ satisfies the geometric condition given in Theorem \ref{thm:main-result}, and the linear dynamical condition \eqref{eq:linear-dynamical-condition}, then $v$ is fixed by the whole group $G$.  Let us fix $s_0 \in I$. By replacing $\varphi(s)$ by $\varphi(s) - \varphi(s_0)$ and $v$ by $u_k(\varphi(s_0))v$, we may assume that $\varphi(s_0) = \vect{0}$ and $v \in V^{-0}(A)$. When $\eta_1 = \eta_2 = \cdots = \eta_k$, then for any $s \in I \setminus \{s_0\}$, one can embed $A$ and $u_k(\varphi(s))$ into an $\SL(2,\R)$ copy in $H$ such that $a(t)$ corresponds to $\begin{bmatrix}e^t & \\ & e^{-t}\end{bmatrix}$ and $u_k(\varphi(s))$ corresponds to $\begin{bmatrix} 1 & 1 \\ 0 & 1\end{bmatrix}$. Then we can apply the basic lemma proved by the author \cite[Lemma 5.1]{Yang_2} on $\SL(2,\R)$ representations and the argument developed in \cite{Yang_2} and \cite{Yang_1} to show that the action of $G$ on $v$ is trivial. For general $\eta_1, \eta_2, \dots, \eta_k$, this argument does not work since we can not embed $A$ and $u_k(\varphi(s))$ into an $\SL(2, \R)$ copy in $H$. This is the main difficulty to complete the proof and we need a new idea to overcome the difficulty. 
 \par The new idea we will develop in this paper is called {\em perturbation of the diagonal subgroup}. The main idea goes as follows. We want to show that under some geometric condition on $\varphi$, if $u_k(\varphi(s))v \in V^{-0}(A)$ for all $s \in I$, then $u_k(\varphi(s))v \in V^{-0}(A')$ for all $s \in I$, where $A'$ is another diagonal subgroup such that $A'$ and $u_k(\varphi(s))$ can be embedded into an $\SL(2, \R)$ copy in $H$. This allows us to apply the previous argument to show that $v$ is fixed by $G$.
 \par If $H$ is simple, then by a standard argument, one can conclude Theorem \ref{thm:main-result}. However, in our case, $H = (\SO(n,1))^k$ is not simple, so we could not get the equidistribution result. This turns out to be the second difficulty of the proof. We then proceed as follows. By the linearization technique, $G v = v$ implies that every $W$-ergodic component of $\lambda_{\infty}$ is of form $g'\mu_{G_1}$ where $G_1$ is a normal subgroup of $G_1$ such that $G_1\cap \Gamma$ is a lattice of $G_1$, and $\mu_{G_1}$ denotes the unique $G_1$ invariant probability measure on $G_1\Gamma$ induced by the Haar measure on $G_1$. Moreover, $H_1 = G_1 \cap H$ is a normal subgroup of $H$ containing $W$. We then consider the quotient group $ G(1) := G/G_1$ with the lattice $\Gamma(1) : = \Gamma/ G_1\cap \Gamma$. Let $X(1):= G(1)/\Gamma(1)$. Then we can push each $\lambda_t$ to a probability measure $\lambda_{t}(1)$ on $X(1)$, and show that any accumulation point $\lambda_{\infty}(1)$ of $\{\lambda_{t}(1): t >0\}$ is a probability measure on $X(1)$ invariant under a unipotent subgroup $W(1)$ of $H(1) = H/H_1$. By pulling $W(1)$ back to the original group $H$, we can deduce that any accumulation point $\lambda_{\infty}$ of $\{\lambda_t: t >0\}$ is invariant under another unipotent subgroup $W'$ of $H$. By applying the linearization technique with $W$ replaced by $W'$, we can get a strictly larger normal subgroup $G'_1$ with the properties as above. Repeating this process, we can finally get that $G_1 = G$, which implies Theorem \ref{thm:main-result}. 
 \subsection{Organization of the paper}
 \label{subsection:organization-of-the-paper}
 The paper is organized as follows. In \S \ref{section:unipotent-invariance-non-divergence}, as explained above, we will modify $\mu_t$ to $\lambda_t$, and show that any accumulation point $\lambda_{\infty}$ of $\{\lambda_t : t > 0\}$ is a probability measure on $X$ invariant under the action of some unipotent subgroup $W \subset H$. In \S \ref{section:ratner-thm-linearization-technique}, we will apply the Ratner's theorem and the linearization technique to deduce the linear dynamical condition described above. In \S \ref{section:basic-lemmas}, we will prove a basic lemma that allows us to proceed the {\em perturbation of the diagonal subgroup} described as above. In \S \ref{section:normal-subgroups}, we will study the case that every $W$-ergodic component of $\lambda_{\infty}$ is induced by the Haar measure on some normal subgroup $G_1$ of $G$ with closed orbit. In \S \ref{section:conclusion}, we will finish the proof of Theorem \ref{thm:main-result}.
 \begin{notation}
 Throughout this paper, we will use the following notation. 
 \par For $\epsilon > 0$ small, 
 and two quantities $A$ and $B$, $A \overset{\epsilon}{\approx} B$ means that 
 $|A-B| \leq \epsilon$. Fix a right $G$-invariant metric $d(\cdot, \cdot)$ on $G$, then for $x_1, x_2 \in X = G/\Gamma$,
 and $\epsilon >0$, $x_1 \overset{\epsilon}{\approx} x_2$ means $x_2 =g x_1$ such that $d(g,e)< \epsilon$. 
 \par Given some quantity $A >0$, $O(A)$ denotes some quantity $B$ such that $|B| \leq C A$ for some contant $C >0$. 
 \end{notation}
 \subsection*{Acknowledgements} 
 The author thanks Professor Nimish Shah for suggesting this problem to him.

\section{Unipotent invariance and non-divergence of limit measures}
\label{section:unipotent-invariance-non-divergence}

\subsection{Preliminaries on Lie group structures}
\label{subsection:preliminaries}
We first recall some basic facts of the groups $\SO(n,1)$ and $H = (\SO(n,1))^k$. We may realize $\SO(n,1)$ as the group of $n+1$ by $n+1$ matrices with determinant one and preserving the quadratic form $Q$ in $n+1$ real variables defined as follows:
$$Q(x_0, x_1, \dots, x_n) := 2x_0 x_n - (x_1^2 + \cdots + x_{n-1}^2).$$
Let us fix the maximal $\R$ split Cartan subgroup $B=\{b(t): t \in \R\}$ as follows:
$$b(t) : = \begin{bmatrix}e^{2t} & & \\ & \I_{n-1} & \\ & & e^{-2t}\end{bmatrix}.$$
Let us denote $B^+ = \{b(t): t >0\}$. The expanding horospherical subgroup $U^+(B)$ of $\SO(n,1)$ with respect to $B$ is isomorphic to $\R^{n-1}$ with isomorphism $u: \R^{n-1} \rightarrow U^+ (B)$ defined as follows:
$$u(\vect{x}) := \begin{bmatrix} 1 & x_1  & \cdots & x_{n-1} &  \|\vect{x}\|^2_2 \\ & 1 &  & & 2 x_1 \\ & & \ddots & & \vdots \\  &  & & 1 & 2 x_{n-1} \\ &  & & & 1 \end{bmatrix}$$
for $\vect{x} = (x_1, x_2, \dots, x_{n-1}) \in \R^{n-1}$. Similarly the contracting horospherical subgroup $U^-(B)$ of $\SO(n,1)$ with respect to $B$ is isomorphic to $\R^{n-1}$ with isomorphism $u^-: \R^{n-1} \rightarrow U^-(B)$ defined as follows:
$$u^- (\vect{x}) := \begin{bmatrix}1 & & & &  \\ 2 x_1 & 1 & & &   \\ \vdots & & \ddots & & \\ 2 x_{n-1}  & & & 1 & \\ \|\vect{x}\|^2_2 & x_1 & \cdots  & x_{n-1} & 1 \end{bmatrix}$$
for $\vect{x} = (x_1, x_2, \dots, x_{n-1})\in \R^{n-1}$.
\par Let $Z(B) \subset \SO(n,1)$ denote the centralizer of $B$ in $\SO(n,1)$. We may choose a maximal compact subgroup $K \subset \SO(n,1)$ isomorphic to $\SO(n)$ such that $M := Z(B)\cap K $ is of form as follows:
$$M = \left\{ m := \begin{bmatrix}1 & & \\ & k(m) & \\ & & 1 \end{bmatrix}: k(m) \in \SO(n-1) \right\}.$$
By identifying $m \in M$ with $k(m) \in \SO(n-1)$, we may identify $M$ with $\SO(n-1)$. We also have that $Z(B) = MB$. 
\par It is easy to check that for any $m \in M = \SO(n-1)$ and any $\vect{x}\in \R^{n-1}$, $m u(\vect{x}) m\inv = u(m\cdot \vect{x})$ and $m u^-(\vect{x}) m\inv = u^-(m \cdot \vect{x})$ where $m \cdot \vect{x}$ denotes the standard action of $\SO(n-1)$ on $\R^{n-1}$. For $t \in \R$, we have that $b(t)u(\vect{x}) b(-t) = u(e^{2t}\vect{x})$ and $b(t) u^-(\vect{x}) b(-t) = u^- (e^{-2t}\vect{x})$. Therefore, the conjugate action of $Z(B) = MB$ on $U^{+}(B)$ defines an action of $Z(B)$ on $\R^{n-1}$. Let us denote this action by $\Ad^+$. For the same reason, the conjugate action of $Z(B)$ on $U^-(B)$ defines an action of $Z(B)$ on $\R^{n-1}$. We denote this action by $\Ad^-$. It is easy to check that for $m b(t) \in Z(B)$ where $m \in \SO(n-1)$, and $\vect{x} \in \R^{n-1}$, $\Ad^+_{m b(t)} (\vect{x}) = e^{2t} m \cdot \vect{x}$ and $\Ad^-_{m b(t)} (\vect{x}) = e^{-2t} m \cdot \vect{x}$.
\par For any $\vect{x} \in \R^{n-1}\setminus \{\vect{0}\}$, there exists a $\SL(2,\R)$ copy in $\SO(n,1)$ containing $\{u(r \vect{x}): r \in \R\}$, $B$ and $\{u^-(r \vect{x}): r \in \R\}$. In this $\SL(2, \R)$ copy, $u(\vect{x})$ corresponds to $\begin{bmatrix}1 & 1 \\ 0 & 1 \end{bmatrix}$, $b(t)$ corresponds to $\begin{bmatrix}e^t & \\ & e^{-t}\end{bmatrix}$, and $u^-(\vect{x}\inv)$ corresponds to $\begin{bmatrix}1 & 0\\ 1 & 1 \end{bmatrix}$ (recall that for $\vect{x} \in \R^{n-1}$, $\vect{x}\inv := \frac{\vect{x}}{\|\vect{x}\|_2^2}$). Let us denote this $\SL(2, \R)$ copy by $\SL(2, \vect{x})$. Let us define
$$E := \begin{bmatrix} -1 &  \\ & \I_{n-2} \end{bmatrix},$$
and 
$$J(\vect{x}) := \begin{bmatrix} & & \|\vect{x}\|_2^2 \\ & E & \\ \|\vect{x}\|_2^{-2} & & \end{bmatrix}.$$
It is easy to check that $J(\vect{x}) \in \SL(2, \vect{x})$ and it corresponds to $\begin{bmatrix}0 & 1 \\ -1 & 0\end{bmatrix} \in \SL(2, \R)$.
\par Now let us have a look at $H = (\SO(n,1))^k$. Recall that 
$$A = \{a(t) = (b(\eta_1 t), b(\eta_2 t), \dots , b(\eta_k t)): t \in \R \}.$$
 Let $U^+(A) \subset H$ and $U^- (A) \subset H$ denote the expanding and contracting horospherical subgroups of $H$ with respect to $A$ respectively. Let us define:
 $$u_k(\vect{x}_1, \vect{x}_2, \dots, \vect{x}_k) := (u(\vect{x}_1), u(\vect{x}_2), \dots, u(\vect{x}_k))$$
 and 
 $$u^-_k(\vect{x}_1, \vect{x}_2, \dots, \vect{x}_k) := (u^-(\vect{x}_1), u^-(\vect{x}_2), \dots, u^-(\vect{x}_k))$$
 for $(\vect{x}_1, \vect{x}_2, \dots, \vect{x}_k)\in (\R^{n-1})^k$. It is easy to check that $u_k : (\R^{n-1})^k \rightarrow U^+(A)$ and $u^-_k : (\R^{n-1})^k \rightarrow U^-(A)$ both give isomorphisms. The centralizer of $A$ in $H$, $Z(A) = (\SO(n-1))^k \cdot B^k$. The conjugate action of $Z(A)$ on $U^+(A)$ (and $U^- (A)$, respectively) defines an action of $Z(A)$ on $(\R^{n-1})^k$ which we will denote by $\Ad^+$ (and $\Ad^-$, respectively). It is easy to check that for $\mathfrak{b} = (b(t_1), b(t_2) , \dots , b(t_k)) \in B^k$, $\mathfrak{m} = (m_1, m_2,\dots, m_k) \in (\SO(n-1))^k$, and $(\vect{x}_1, \vect{x}_2, \dots, \vect{x}_k) \in (\R^{n-1})^k$, 
 $$\Ad^+_{\mathfrak{b} \mathfrak{m}}(\vect{x}_1, \vect{x}_2, \dots, \vect{x}_k) = (e^{2t_1} m_1 \cdot \vect{x}_1, e^{2t_2} m_2 \cdot \vect{x}_2, \dots, e^{2 t_k} m_k \cdot \vect{x}_k),$$
 and 
 $$\Ad^-_{\mathfrak{b} \mathfrak{m}}(\vect{x}_1, \vect{x}_2, \dots, \vect{x}_k) = (e^{-2t_1} m_1 \cdot \vect{x}_1, e^{-2t_2} m_2 \cdot \vect{x}_2, \dots, e^{-2 t_k} m_k \cdot \vect{x}_k).$$
  For any $(\vect{x}_1, \vect{x}_2, \dots, \vect{x}_k) \in (\R^{n-1})^k$, $u(\vect{x}_1, \vect{x}_2, \dots, \vect{x}_k)$ can be embedded into an $\SL(2, \R)$ copy in $H$ such that in this copy, $u(\vect{x}_1, \vect{x}_2, \dots, \vect{x}_k)$ corresponds to $\begin{bmatrix}1 & 1 \\ 0 & 1\end{bmatrix}$, $d(t):= (b(t), b(t),\dots, b(t))$ corresponds to $\begin{bmatrix}e^t &  \\ & e^{-t}\end{bmatrix}$, and $u^{-}(\vect{x}_1\inv, \vect{x}_1\inv, \dots, \vect{x}_k\inv)$ corresponds to $\begin{bmatrix}1 & 0 \\ 1 & 1\end{bmatrix}$. Let us denote this $\SL(2, \R)$ copy by $\SL(2, \vect{x}_1, \vect{x}_2, \dots, \vect{x}_k)$. Note that
  $$J(\vect{x}_1, \vect{x}_2, \dots, \vect{x}_k) := (J(\vect{x}_1), J(\vect{x}_2), \dots , J(\vect{x}_k)) \in \SL(2, \vect{x}_1, \vect{x}_2, \dots, \vect{x}_k)$$ 
  and it corresponds to $\begin{bmatrix} 0 & 1 \\ -1 & 0\end{bmatrix}$ in $\SL(2, \R)$. 


\subsection{Unipotent invariance of limit measures}
\label{subsection:unipotent-invariance}
Let us put
 $$\varphi(s) = (\varphi^1(s), \varphi^2(s), \dots , \varphi^k(s))$$ 
 where $\varphi^i(s) \in \R^{n-1}$ denotes the $i$th component of $\varphi$. Recall that for $t >0$, $\mu_t$ denotes the normalized parametric measure on the curve $a(t)u_k(\varphi(I))x$, namely, for 
$f \in C_c(X)$, 
$$\int f \dd \mu_t = \frac{1}{|I|} \int_a^b f(a(t)u_k(\varphi(s))x) \dd s.$$
Given a subinterval $J \subset I$, let $\mu_t^J$ denote the normalized parametric measure on the curve $a(t)u_k(\varphi(J))x$. In this subsection, we will modify $\mu_t$'s and prove the unipotent invariance for the corresponding limit measures. 
\par Without loss of generality, we may assume that $\eta_1 \geq \eta_2 \geq \cdots \geq \eta_k >0$. Moreover, we may assume that $\eta = \eta_1 =  \cdots = \eta_{k_1} > \eta_{k_1 +1}$. Since for $s_0 \in I$, $(\varphi(s) -\varphi(s_0))\inv$ is defined for almost every $s \in I$, we have that $\varphi^i(s)$ is not constant for each $i=1,2,\dots, k$. Therefore, there is a finite set $S = \{s_1, s_2, \dots, s_l\} \subset I$, such that for any $s \in I\setminus S$, $(\varphi^i)^{(1)}(s) \neq \vect{0}$ for $i=1,2,\dots, k$. Here $(\varphi^i)^{(1)}$ denotes the derivative of $\varphi^i$. 
\begin{definition}
\label{def:lambda-t}
\par For any closed subinterval $J \subset I \setminus S$, we can choose an analytic curve $z: J \rightarrow Z(A)$ such that for any $s \in J$,
$\Ad^+_{z(s)} \varphi^{(1)}(s) = \mathfrak{e} := (\vect{e}_1 , \vect{e}_1, \dots, \vect{e}_1)$ where $\vect{e}_1 = (1,0,\dots, 0) \in \R^{n-1}$.
For any closed subinterval $J \subset I \setminus S$ and $t >0$, we define $\lambda^J_t$ to be the normalized parametric measure on the curve $\{z(s) u_k(\varphi(s)) x : s \in J\}$; namely, for $f \in C_c(X)$,
$$\int f \dd \lambda_t^J := \frac{1}{|J|} \int_{s\in J} f(z(s)u_k(\varphi(s))x) \dd s .$$
\end{definition}
\begin{remark}
\label{rmk:lambda-t}
The definition of $\lambda^J_t$ is due to Nimish Shah \cite{Shah_1} and \cite{Shah_2}.
\end{remark}
\begin{proposition}
\label{prop:lambda-to-mu}
Suppose for any closed subinterval $J \subset I \setminus S$, $\lambda^J_t \rightarrow \mu_G$ as $t \rightarrow +\infty$, then 
$\mu_t \rightarrow \mu_G$ as $t \rightarrow +\infty$.
\end{proposition}
\begin{proof}
For any fixed $f \in C_c(G/\Gamma)$ and $\epsilon >0$, we want to show that for $t >0$ large enough, 
   $$\int f \dd \mu_t \overset{4 \epsilon}{ \approx} \int_{X} f \dd \mu_G.$$
 For each $s' \in S$, one can choose an open subinterval $B_{s'} \subset I$ containing $s'$, such that 
 \begin{equation}
 \label{equ:bs-1}
 \left| (\sum_{s'\in S} |B_{s'}|) \int_{X} f \dd \mu_G \right| \leq \epsilon |I|,
 \end{equation}
 and for any $t >0$,
 \begin{equation}
 \label{equ:bs-2}
 \left| \int_{\cup_{s'\in S} B_{s'}} f(a(t) u_k (\varphi(s)) x) \dd s \right| \leq \epsilon |I|.
 \end{equation}
 Since $f$ is uniformly continuous, there exists a constant $\delta>0$, such that if $x_1 \overset{\delta}{\approx} x_2$
 then $f(x_1) \overset{\epsilon}{\approx} f(x_2)$.
 \par We cut $I\setminus \cup_{s'\in S} B_{s'}$ into several small closed subintervals $J_1, J_2, \dots, J_p$, such that, for every $J_r$, $z^{-1}(s_1)z(s_2) \overset{\delta}{\approx} e$ for any $s_1,s_2 \in J_r$. 
 \par Now for a fixed $J_r \subset I\setminus \cup_{s'\in S} B_{s'}$, we choose $s_0 \in J_r$ and define $f_0(x)= f(z^{-1}(s_0)x)$. Then for any $s \in J_r$, because 
 $z^{-1}(s_0)z(s)a(t)u_k(\varphi(s))x \overset{\delta}{\approx} a(t)u_k(\varphi(s))x$, we have
 $$f_0(z(s)a(t)u_k(\varphi(s))x) = f(z^{-1}(s_0)z(s)a(t)u_k(\varphi(s))x) \overset{\epsilon}{\approx} f(a(t)u_k(\varphi(s))x).$$
 Therefore 
 $$\int f_0 \dd \lambda^{J_r}_t \overset{\epsilon}{\approx} \int f \dd \mu^{J_r}_t.$$
 Because $\int f_0 \dd \lambda^{J_r}_t \rightarrow \int_{X} f_0(x) \dd\mu_G(x)$ as $t \rightarrow +\infty$, and 
 $\int_{X} f_0(x) \dd\mu_G(x) = \int_{X}f(z^{-1}(s_0)x) \dd \mu_{G}(x) = \int_{X} f(x) \dd\mu_G$ (because $\mu_G$ is
 $G$-invariant), we have that there exists a constant $T_r>0$, such that for $t> T_r$, 
 $$\int f_0 \dd\lambda^{J_r}_t \overset{\epsilon}{\approx} \int_{X} f \dd\mu_G.$$
 Therefore, for $t > T_r$,
 $$\int f \dd\mu^{J_r}_t \overset{2\epsilon}{\approx} \int_{X} f \dd\mu_G,$$
i.e.,
$$\int_{J_r} f(a(t)u_k(\varphi(s)) x ) \dd s \overset{2 \epsilon |J_r|}{ \approx} |J_r| \int_{X} f \dd \mu_G .$$
 Then for $t > \max_{1\leq r\leq p} T_r$, we can sum up the above approximations for $r=1,2,\dots, p$ and get 
 $$\int_{I \setminus \cup_{s' \in S} B_{s'} } f(a(t) u_k(\varphi(s))x) \dd s \overset{2\epsilon |I|}{\approx} (|I| - \sum_{s' \in S} |B_{s'}|)\int_{X} f \dd \mu_G.$$
Combined with \eqref{equ:bs-1} and \eqref{equ:bs-2}, the above approximation implies that 
$$\int_I f(a(t) u_k(\varphi(s))x) \dd s \overset{4\epsilon |I|}{ \approx} |I| \int_{X} f \dd \mu_G ,$$
which is equivalent to 
$$\int f \dd \mu_t  \overset{4\epsilon}{\approx} \int_{X} f \dd \mu_G .$$
 Because $\epsilon>0$ can be arbitrarily small, the proof is completed.
\end{proof}
By Proposition \ref{prop:lambda-to-mu}, to prove $\mu_t \rightarrow \mu_G $ as $t \rightarrow +\infty$, it suffices to show that for any closed subinterval $J \subset I\setminus S$, $\lambda^{J}_t \rightarrow \mu_G$ as $t \rightarrow +\infty$. In particular, if we can prove the equidistribution of $\{\lambda_t := \lambda^I_t : t >0\}$ as $t \rightarrow +\infty$ assuming that for all $s\in I$, $(\varphi^i)^{(1)}(s) \neq \vect{0}$ for $i = 1,2,\dots, k$, then the equidistribution of $\{\mu_t : t >0\}$ as $t \rightarrow +\infty$ will follow. Therefore, later in this paper, we will assume that for all $s \in I$, $(\varphi^i)^{(1)}(s) \neq \vect{0}$ for $i = 1,2,\dots, k$, and define $\lambda_t$ to be the normalised Lebesgue measure on the curve $\{z(s)a(t)u_k(\varphi(s))x: s \in I\}$. Our goal is to show that $\lambda_t \rightarrow \mu_G$ as $t \rightarrow +\infty$.
\par For $k' \leq k$, let $\mathfrak{e}_{k'} = (\vect{e}_1, \dots, \vect{e}_1, \vect{0}, \dots, \vect{0})$ where the first $k'$ components are $\vect{e}_1$ and the rest are $\vect{0}$. Let 
\begin{equation} 
\label{eq:W}
W_{k'}:= \{u_k(r\mathfrak{e}_{k'}): r \in \R\}.
\end{equation}
 We will show that any limit measure of $\{\lambda_t : t >0\}$ is invariant under the unipotent subgroup $W_{k_1}$.

\begin{proposition}[See \cite{Shah_1}]
 \label{prop:unipotent-invariance}
 Let $t_i \rightarrow +\infty$ be a sequence such that $\lambda_{t_i} \rightarrow \lambda_{\infty}$ in weak-$\ast$ topology, then 
 $\lambda_{\infty}$ is invariant under the $W_{k_1}$-action.
\end{proposition}
\begin{proof}
Given any $f\in C_c(X)$, and $r \in \R$, we want to show that 
$$\int f(u_k(r \mathfrak{e}_{k_1}) x) \dd \lambda_{\infty} = \int f(x) \dd \lambda_{\infty}.$$
Since $z(s)$ and $\varphi(s)$ are analytic and defined on closed interval $I = [a, b]$, there exists a constant $T_1 >0$ such that for $t \geq T_1$, $z(s)$ and $\varphi(s)$ can be extended to analytic curves defined on $[a - |r| e^{-2\eta t} , b + |r| e^{-2\eta t}]$. Throughout the proof, we always assume that $t_i \geq T_0$. Then $z(s + r e^{-2 \eta t_i})$ and $\varphi(s + r e^{-2\eta t_i})$ are both well defined for $s \in I$.
\par From the definition of $\lambda_{\infty}$ we have 
$$
\int f(u_k(r\mathfrak{e}_{k_1})x)\dd\lambda_{\infty}  =  \lim_{t_i \rightarrow +\infty} \frac{1}{|I|}\int_{s \in I} f(u_k(r\mathfrak{e}_{k_1})z(s)a(t_i)u_k(\varphi(s))x) \dd s.
$$
We want to show that 
$$u_k(r\mathfrak{e}_{k_1})z(s)a(t_i)u_k(\varphi(s)) \overset{O(e^{-2(\eta - \eta_{k_1+1} ) t_i})}{\approx} z(s+r e^{-2\eta t_i})a(t_i)u_k(\varphi(s+r e^{-2\eta t_i})).$$

Since $z(s+r e^{-2\eta t_i}) \overset{O(e^{-2\eta t_i})}{\approx} z(s)$ for $t_i$ large enough, it suffices to show that 
$$u(r\mathfrak{e}_{k_1})z(s)a(t_i)u_k(\varphi(s)) \overset{O(e^{-2(\eta - \eta_{k_1 + 1} ) t_i})}{\approx} z(s)a(t_i)u_k(\varphi(s+r e^{-2\eta t_i})).$$
In fact,
$$\begin{array}{cl}
   & z(s)a(t_i)u_k(\varphi(s+r e^{-2\eta t_i})) \\
   = & z(s) a(t_i)u_k(\varphi(s) + r e^{-2\eta t_i} \varphi^{(1)}(s) + O(e^{-4\eta t_i}) ) \\
   = &  u_k( O(e^{- 2\eta t_i}) ) a(t_i)z(s)u_k(r e^{-2 \eta t_i} \varphi^{(1)}(s)) u_k(\varphi(s)).
  \end{array}
$$
By the definition of $z(s)$, we have the above is equal to
$$\begin{array}{cl}  & u_k(O(e^{-2\eta t_i}))a(t_i)u_k(r e^{-2\eta t_i} \mathfrak{e})z(s)u_k(\varphi(s))  \\  = & u_k(O(e^{-2\eta t_i}))u_k(O(e^{-2(\eta - \eta_{k_1+1}) t_i}))u_k(r \mathfrak{e}_{k_1})z(s)a(t_i)u_k(\varphi(s)) \\  = & u_k(O(e^{-2(\eta -\eta_{k_1+1} ) t_i}))u_k(r \mathfrak{e}_{k_1})z(s)a(t_i)u_k(\varphi(s)) . \end{array}$$

 This shows that 
  
$$u_k(r\mathfrak{e}_{k_1} )z(s)a(t_i)u_k(\varphi(s)) \overset{O(e^{-2(\eta - \eta_{k_1} ) t_i})}{\approx} z(s+ r e^{-2\eta t_i})a(t_i)u_k(\varphi(s+r e^{-2\eta t_i})).$$

Therefore, for any $\delta >0$, there is some constant $T \geq T_1$, such that for $t_i \geq T$,

$$u_k(r\mathfrak{e}_{k_1} )z(s)a(t_i)u_k(\varphi(s)) \overset{\delta}{\approx} z(s+ r e^{-2\eta t_i})a(t_i)u_k(\varphi(s+r e^{-2\eta t_i})).$$

Given $\epsilon >0$, we choose $\delta>0$ such that whenever $x_1 \overset{\delta}{\approx} x_2$, we have that
$f(x_1) \overset{\epsilon}{\approx} f(x_2)$. Let $T>0$ be the constant as above. Then from the above argument, for $t_i > T$, we have
$$f(u_k(r \mathfrak{e}_{k_1} ) z(s) a(t_i)u_k(\varphi(s))x) \overset{\epsilon}{\approx} f(z(s+ r e^{-2\eta t_i}) a(t_i) u_k(\varphi(s+ r e^{-2\eta t_i}))x),$$
therefore,
$$\begin{array}{cl} & \frac{1}{|I|}\int_{s\in I} f(u_k(r \mathfrak{e}_{k_1} ) z(s) a(t_i)u_k(\varphi(s))x)\dd s \\
                  \overset{\epsilon}{\approx} & \frac{1}{|I|} \int_{s\in I} f(z(s+ r e^{-2\eta t_i}) a(t_i) u_k(\varphi(s+ r e^{-2\eta t_i}))x) \dd s \\
                  = & \frac{1}{|I|} \int_{a+ r e^{-2\eta t_i}}^{b+ r e^{-2\eta t_i}} f(z(s)a(t_i)u_k(\varphi(s))x) \dd s.
\end{array}$$

It is easy to see that when $t_i$ is large enough,
$$\frac{1}{|I|} \int_{a+ r e^{-2\eta t_i}}^{b+ r e^{-2\eta t_i}} f(z(s)a(t_i)u_k(\varphi(s))x) \dd s \overset{\epsilon}{\approx} \frac{1}{|I|} \int_{a}^{b} f(z(s)a(t_i)u_k(\varphi(s))x) \dd s.$$
Therefore, for $t_i$ large enough,
$$\int f(u_k(r \mathfrak{e}_{k_1})x) \dd \lambda_{t_i} \overset{2\epsilon}{\approx} \int f(x) \dd \lambda_{t_i}.$$
Letting $t_i \rightarrow +\infty$, we have 
$$\int f(u_k(r \mathfrak{e}_{k_1})x) \dd \mu_{\infty} \overset{2\epsilon}{\approx} \int f(x) \dd \mu_{\infty}.$$
Since the above approximation is true for arbitrary $\epsilon >0$, we have that 
$\mu_{\infty}$ is $W_{k_1}$-invariant.
\end{proof}
For simplicity, later in this paper, we will denote $W_{k_1}$ by $W$.
\subsection{Non-divergence of limit measures} 
\label{subsection:non-divergence}
We will prove that any limit measure $\lambda_{\infty}$ of $\{\lambda_t: t >0\}$ is still a probability measure on $X$. To do this, it suffices to prove the following proposition.
\begin{proposition}
\label{prop:non-divergence}
For any $\epsilon >0$, there exists a compact subset $\mathcal{K}_{\epsilon} \subset X$ such that 
$$\lambda_{t}(\mathcal{K}_{\epsilon}) \geq 1-\epsilon \text{ for all } t > 0. $$
\end{proposition}
To prove this proposition, we need to introduce a particular representation of $G$.
\begin{definition}
\label{def:representation-G}
Let $\mathfrak{g}$ denote the Lie algebra of $G$, and let $d = \dim G$. We define
 $$V := \bigoplus_{i=1}^{d} \bigwedge\nolimits^{i} \mathfrak{g}.$$
 Let $G$ act on $V$ via $\bigoplus_{i=1}^d \bigwedge ^{i} \mathrm{Ad}(\cdot)$. This defines a linear representation of $G$:
 $$G \rightarrow \mathrm{GL}(\mathcal{V}).$$
\end{definition}
\begin{remark}
In this paper, we will treat $V$ as a representation of $H$.
\end{remark}
\begin{notation}
\label{notation:decomposition-eigenspace}
\par Let $F$ be a Lie group, and $\mathcal{V}$ be a finite dimensional linear representation of $F$. Then for a one-parameter diagonal subgroup 
 $D=\{d(t): t \in \R\}$ of $F$, we can decompose $\mathcal{V}$ as the direct sum of eigenspaces of $D$, i.e.,
 $$\mathcal{V} = \bigoplus_{\lambda \in \R} \mathcal{V}^{\lambda}(D),$$
 where $\mathcal{V}^{\lambda}(D) = \{v\in \mathcal{V}: d(t)v = e^{\lambda t} v\}$.
 \par We define
 $$\mathcal{V}^{+}(D) = \bigoplus_{\lambda >0} \mathcal{V}^{\lambda}(D),$$
 $$\mathcal{V}^{-}(D) = \bigoplus_{\lambda <0 } \mathcal{V}^{\lambda}(D),$$
 and similarly,
 $$\mathcal{V}^{\pm 0}(D) = \mathcal{V}^{\pm}(D) + \mathcal{V}^0(D).$$
 For a vector $v \in \mathcal{V}$, we denote by $v^{+}(D)$ ($v^{-}(D)$, $v^0(D)$ and $v^{\pm 0}(D)$ respectively) the projection of $v$ onto $\mathcal{V}^{+}(D)$ ($\mathcal{V}^{-}(D)$, $\mathcal{V}^0(D)$, and $\mathcal{V}^{\pm 0}(D)$ respectively).
\end{notation}
The following theorem due to Kleinbock and Margulis \cite{Klein_Mar} is the basic tool to prove the proposition:
\begin{theorem}[see \cite{Dani} and \cite{Klein_Mar}]
 \label{thm:non-divergence}
 Fix a norm $\|\cdot \|$ on $V$. There exist finitely many vectors  $v_1, v_2, \dots , v_r \in V$ such that for each 
 $i=1,2,\dots, r$, the orbit $\Gamma v_i$ is discrete, and the
following holds: for any $\epsilon >0$ and $R > 0$, there exists a
compact set $K\subset X$ such that for any $t >0$ and any subinterval $J\subset I$, one of the
following holds:
\begin{enumerate}[label=\textbf{S.\arabic*}]
\item There exist $\gamma \in \Gamma$ and $j\in \{1,\dots , r\}$
such that $$\sup_{s\in J} \| a(t)u_k(\varphi(s)) g \gamma v_j \| < R.$$
\item $$|\{ s\in J:  a(t)u_k(\varphi(s))x \in K\}| \geq (1-\epsilon)|J|.$$
\end{enumerate}
\end{theorem}
\begin{remark}
 The proof for polynomial curves is due to Dani \cite{Dani}, 
 and the proof for analytic curves is due to Kleinbock and Margulis \cite{Klein_Mar}. The crucial part to prove the above theorem is to find constants 
 $C>0$ and $\alpha>0$ such that in this particular representation, all the coordinate functions of $a(t)u_k(\varphi(\cdot))$ 
 are $(C, \alpha)$-good. Here a function $f: I \rightarrow \R$ is called
 $(C,\alpha)$-good if for any subinterval $J \subset I$ and any $\epsilon >0$, the following holds:
 $$|\{s\in J: |f(s)|<\epsilon\}| \leq C\left(\frac{\epsilon}{\sup_{s\in J}|f(s)|}\right)^{\alpha} |J|.$$
\end{remark}
\par The following basic lemma is the key to prove Proposition \ref{prop:non-divergence}:
\begin{lemma}[Basic Lemma]
\label{lemma:basic-lemma-1}
Let $V$ be a finite dimensional represenation of $H$, and let $\varphi: I \rightarrow (\R^{n-1})^k$ be an analytic curve as above. For any nonzero vector $v \in V$, there exists some $s \in I$ such that 
$$u_k(\varphi(s))v \not\in V^-(A).$$
Here $V^-(A)$ is defined as in Notation \ref{notation:decomposition-eigenspace}.
\end{lemma}
\par We will postpone the proof to \S \ref{section:basic-lemmas}.
\begin{proof}[Proof of Proposition \ref{prop:non-divergence} assuming Lemma \ref{lemma:basic-lemma-1}]
\par Let $V$ be as in Definition \ref{def:representation-G}. Since $A \subset H$ is a diagonal subgroup, we have the following decomposition as in Notation \ref{notation:decomposition-eigenspace}:
$$V = \bigoplus_{\lambda \in \R} V^{\lambda}(A).$$
 Choose the norm $\|\cdot\|$ on $V$ to be the maximum norm associated to some choices of norms on $V^{\lambda}(A)$'s.
 \par For contradiction we assume that there exists a constant $\epsilon>0$ such that for any compact subset $\mathcal{K} \subset G/\Gamma$, there exist some $t > 0$ such that $\lambda_{t}(\mathcal{K}) <1 - \epsilon$. Now we fix a sequence $\{R_i> 0: i \in \N \}$ tending to zero. By Theorem \ref{thm:non-divergence}, for any $R_i$, there exists a 
 compact subset $\mathcal{K}_i\subset X$, such that for any $t >0$, one of the
following holds:
\begin{enumerate}[label=\textit{S\arabic*.}]
\item \label{divergence_condition} There exist $\gamma \in \Gamma$ and $j\in \{1,\dots , r\}$
such that $$\sup_{s\in I} \| a(t)u_k(\varphi(s)) g \gamma v_j \| < R_{i},$$
\item \label{non_divergence_condition} $$|\{ s\in I:  a(t)u_k(\varphi(s))x \in \mathcal{K}_i\}| \geq (1-\epsilon)\abs{I}.$$
\end{enumerate}
\par From our hypothesis, for each $\mathcal{K}_i$, there exists some $t_i>0$ such that \ref{non_divergence_condition} does not hold.
So there exist $\gamma_i \in \Gamma$ and $v_{j(i)}$ such that 
\begin{equation} \label{eq:vj}
\sup_{s\in I} \| a(t_i)u_k(\varphi(s)) g \gamma_i v_{j(i)} \| < R_i.
\end{equation}
By passing to a subsequence of $\{i\in \N\}$, we may assume that $v_{j(i)} = v_j$ remains the same for all $i$. 

Since $\Gamma v_j$ is discrete in $V$, we have $t_{i}\to \infty$ as $i\to\infty$ and there are the following two cases:
\begin{enumerate}[label=\textit{Case~\arabic*.}]
\item \label{case_1} By passing to a subsequence of $\{i\in \N\}$, $\gamma_i v_j= \gamma v_j$ remains the same for all $i$.
\item \label{case_2} $\|\gamma_i v_j\| \rightarrow \infty $ along some subsequence.
\end{enumerate}
\par For Case~1: We have $a(t_i) u_k(\varphi(s)) g \gamma v_j \rightarrow \vect{0}$ as $i \rightarrow \infty$ for all $s \in I$. This implies that 
$$\{u_k(\varphi(s)) g \gamma v_j\}_{s\in I} \subset V^{-}(A),$$
which contradicts Lemma \ref{lemma:basic-lemma-1}.
\par For Case~2: After passing to a subsequence, we have
 \begin{equation} \label{eq:v}
 v := \lim_{i\rightarrow \infty} g \gamma_i v_j/{\norm{g \gamma_i v_j}}, \quad
\norm{v}=1\text{, and } \lim_{i\to\infty}\norm{g \gamma_i v_j} = \infty.
\end{equation}
By Lemma~\ref{lemma:basic-lemma-1}, let $s\in I$ be such that $u_k(\varphi(s))v\not\in V^{-}(A)$. Then by \eqref{eq:v} there exists $\delta_{0}>0$ and $i_{0}\in\N$ such that 
\[
\norm{(u_k(\varphi(s))g \gamma_i v_j)^{0+}}\geq \delta_{0} \norm{g \gamma_i v_j}, \quad \forall i\geq i_{0}.
\]
Then 
\[
\norm{a(t_{i})u_k(\varphi(s))g \gamma_i v_j}\geq \delta_{0} \norm{g \gamma_i v_j}\to\infty,
\text{ as $i\to\infty$},
\] 
which contradicts \eqref{eq:vj}. Thus Cases~1~and~2  both lead to contradictions. 
\end{proof}
\begin{remark}
\label{rmk:non-divergence}
 The same proof also shows that any limit measure of $\{\mu_t : t >0\}$ 
is a probability measure on $X$. 
\end{remark}

\section{Ratner's theorem and the linearization technique}
\label{section:ratner-thm-linearization-technique}
Take any convergent subsequence $\lambda_{t_i} \rightarrow \lambda_{\infty}$. By Proposition \ref{prop:unipotent-invariance} and Proposition \ref{prop:non-divergence}, $\lambda_{\infty}$ is a $W$-invariant probability measure on $X$, where $W$ is a unipotent one-parameter subgroup given by \eqref{eq:W}. We will apply Ratner's theorem and the linearization technique to understand the measure $\lambda_{\infty}$. 

\begin{definition}
\label{def:for-linearization}
 Let $\mathcal{L}$ be the collection of proper analytic subgroups $L < G$ such that $L\cap \Gamma$ is a lattice of $L$. Then $\mathcal{L}$ is a countable set (\cite{Ratner}).
\par For $L\in \mathcal{L}$, define:
$$N(L,W):= \{g\in G: g^{-1}Wg\subset L\}  
\text{, and }
S(L,W):= \bigcup_{L'\in \mathcal{L}, L' \subsetneq L} N(L', W).$$
\end{definition}
\par We formulate Ratner's measure classification theorem as follows (cf. \cite{Mozes_Shah}):
\begin{theorem}[see \cite{Ratner}]
\label{ratner}
Let $\pi: G \rightarrow X =G/\Gamma$ denote the natural projection sending $g\in G$ to $g\Gamma \in X$. Given the $W$-invariant probability measure $\mu$ on
$G/\Gamma$, if $\mu$ is not $G$-invariant then there exists $L \in \mathcal{L}$ such that
\begin{equation}
\begin{array}{ccc}
\mu(\pi(N(L,W)))>0 & \text{ and } & \mu(\pi(S(L,W)))=0.
\end{array}
\end{equation}
Moreover, almost every $W$-ergodic component of $\mu$ on
$\pi(N(L,W))$ is a measure of the form $g\mu_L$ where $g\in
N(L,W)\backslash S(L,W)$, $\mu_L$ is a finite $L$-invariant measure
on $\pi(L)$, and $g\mu_L(E)=\mu_L (g^{-1}E)$ for all Borel sets
$E\subset G/\Gamma$. In particular, if $L\lhd  G$, then the restriction of $\mu$ on $\pi(N(L,W))$ is
$L$-invariant.
\end{theorem}
\par If $\mu_{\infty}=\mu_G$, then there is nothing to prove. So we may assume $\mu_{\infty}\neq \mu_G$. Then by Ratner's Theorem, 
there exists $L \in \mathcal{L}$ such that 
\begin{equation} \label{eq:L}
\mu_{\infty}(\pi(N(L,W))) > 0 \text{ and }  \mu_{\infty}(\pi(S(L,W))) =0.
\end{equation}

Now we want to apply the linearization technique to obtain algebraic consequences of this statement.

\begin{definition}
Let $V$ be the finite dimensional representation of $G$ defined as in Definition \ref{def:representation-G}, for $L \in \mathcal{L}$,
we choose a basis $\mathfrak{e}_1, \mathfrak{e}_2, \dots, \mathfrak{e}_l$ of the Lie algebra $\mathfrak{l}$ of $L$, and define 
$$p_L = \wedge_{i=1}^l \mathfrak{e}_i \in V.$$
Define 
$$\Gamma_L := \left\{\gamma\in \Gamma: \gamma p_L = \pm p_L \right\}.$$
From the action of $G$ on $p_L$, we get a map:
$$\begin{array}{l} \eta:  G \rightarrow V,  \\
                     g \mapsto g p_L .
   \end{array}
$$
Let $\mathcal{A}$ denote the Zariski closure of $\eta(N(L,W))$ in $V$. Then $N(L,W)=G\cap \eta^{-1}({\mathcal A})$. 
\end{definition}
\begin{remark}
\label{rmk:stabilizer-pL}
It is easy to see that the stabilizer of $p_{L}$ is 
\[N_G^1(L) : = \{g \in G: gLg\inv = L \text{, and } \mathrm{det}(\Ad_g |_{\mathfrak{l}}) =1\}.\]
\end{remark}
Using the fact that $\varphi$ is analytic, we obtain the following consequence of the linearization technique (cf.\ \cite{Shah_1,Shah_2,Shah2010}).

\begin{proposition}[{\cite[Proposition~5.5]{Shah_2}}]
\label{prop:relative_small} Let $C$ be a compact subset of $N(H,W)\setminus S(H,W)$. Given $\epsilon > 0$, there exists a
compact set $\mathcal{D}\subset \mathcal{A}$ such that, 
given a relatively compact neighborhood $\Phi$ of
$\mathcal{D}$ in $V$, there exists a neighborhood $\mathcal{O}$
of $C\Gamma$ in $G/\Gamma$ such that for any $t\in\R$ and subinterval
$J\subset I$, one of the following statements holds:
\begin{enumerate}[label=\textit{SS\arabic*.}]
\item \label{1} $|\{s\in J: a(t)u(\varphi(s))g\Gamma \in \mathcal{O}\}|\leq \epsilon \abs{J}$.
\item \label{2} There exists $\gamma\in\Gamma$ such that $\{a(t)z(s)u(\varphi(s))g\gamma p_{L}\}_{s \in I} \subset \Phi$.
\end{enumerate}
\end{proposition}

The following proposition provides the obstruction to the limiting measure not being $G$-invariant in terms of linear actions of groups, and it is a key result for further investigations. 

\begin{proposition}
 \label{prop:algebraic-condition}
 There exists a $\gamma \in \Gamma$ such that
\begin{equation}\label{eq:V0-}
\{u(\varphi(s))g\gamma p_L\}_{s\in I} \subset V^{-0}(A).
\end{equation}
\end{proposition}

\begin{proof}[Proof (assuming Lemma~\ref{lemma:basic-lemma-1}).]
By \eqref{eq:L}, there exists a compact subset $C \subset N(L,W))\setminus S(L,W)$  and $\epsilon>0$ such that $\mu_{\infty}(C\Gamma) >\epsilon>0$. Apply Proposition~\ref{prop:relative_small} to obtain $\mathcal{D}$, and choose any $\Phi$, and obtain a $\mathcal{O}$ so that either \ref{1} or \ref{2} holds.  Since $\lambda_{t_{i}}\not\to \mu_{\infty}$, we conclude that \ref{1} does not hold for $t=t_{i}$ for all $i\geq i_{0}$. Therefore  for every $i\geq i_{0}$, \ref{2} holds and there exists $\gamma_{i}\in\Gamma$ such that
\begin{equation} \label{eq:Phi}
\{a(t_{i})z(s)u(\varphi(s))g\gamma_{i}p_{L}\}_{s \in I}\subset \Phi.
\end{equation}

Since $\Gamma p_{L}$ is discrete in $V$, by passing to a subsequence, there are two cases:
\begin{enumerate}
\item[{\it Case 1.}] $\gamma_{i}p_{L}=\gamma p_{L}$ for some $\gamma\in\Gamma$ for all $i$; or 
\item[{\it Case 2.}] $\norm{\gamma_{i}p_{L}}\to\infty$ as $i\to \infty$. 
\end{enumerate}

In {\it Case~1}, since $\Phi$ is bounded in \eqref{eq:Phi}, we deduce that $z(s)u(\varphi(s))g\gamma p_{L}\subset V^{-0}(A)$ for all $s\in I$. Since $V^{-0}(A)$ is $Z_{H}(A)$-invartiant, \eqref{eq:V0-} holds. 

In {\it Case~2}, by arguing as in the \ref{case_2} of the Proof of Proposition~\ref{prop:non-divergence}, using Lemma~\ref{lemma:basic-lemma-1}, we obtain that $\norm{a(t_{i})u(\varphi(s))g\gamma_{i}p_{L}}\to\infty$. This contradicts \eqref{eq:Phi}, because $z(s)\subset Z_H(A)$ and $\Phi$ is bounded. Thus {\it Case~2} does not occur.
\end{proof}

\section{Basic lemmas}
\label{section:basic-lemmas}
In this section we will recall and prove some basic lemmas on linear representations of $\SL(2,\R)$ and $H = (\SO(n,1))^k$. They are very important to study the linear algebraic condition \eqref{eq:V0-}.
\par We recall the following lemma on linear representations of $\SL(2,\R)$ due to Shah and Yang \cite{shah_yang}:
\begin{lemma}[See~{\cite[Lemma 4.1]{shah_yang}}]
\label{lemma:sl2-representations}
Let $V$ be a finite dimensional linear representation of $\SL(2, \R)$. Let 
$$B = \left\{ b(t):= \begin{bmatrix}e^t & \\ & e^{-t} \end{bmatrix}: t \in \R \right\}, U = \left\{u(s):= \begin{bmatrix}1 & s \\ 0 & 1\end{bmatrix}: s \in \R \right\}, U^- = \left\{u^-(s) :=\begin{bmatrix}1 & 0 \\ s & 1\end{bmatrix}: s \in \R \right\}.$$
Express $V$ as the direct sum of eigenspaces with respect to $B$: 
$$V = \bigoplus_{\omega \in \R} V^{\omega}(B), \text{ where } V^{\omega}(B) : = \{v \in V : b(t)v = e^{\omega t}v \text{ for all }t \in \R \}.$$
For any $v \in V\setminus \{ \vect{0}\}$ and $\omega \in \R$, let $v^{\omega} = v^{\omega}(B)$ denote the $V^{\omega}(B)$-component of $v$, $\omega^{\max}(v) = \max \{\omega: v^{\omega} \neq \vect{0}\}$, and $v^{\max} = v^{\omega^{\max}(v)}$. Then for any $r \neq 0$,
\begin{equation}
\label{eq:omega-max}
\omega^{\max}(u(r)v) \geq - \omega^{\max}(v).
\end{equation}
In particular, if $\omega^{\max}(v) < 0$, then $\omega^{\max}(u(r)v) >0$ for any $r \neq 0$. Moreover if the equality holds in \eqref{eq:omega-max}, then 
\begin{equation}
\label{eq:omega-max-equality}
v = u^-(-r^{-1}) v^{\max} \text{ and } (u(r)v)^{\max} = \sigma(r)v^{\max}, \text{ where } \sigma(r)= \begin{bmatrix}0 & r \\ r\inv & 0 \end{bmatrix}.
\end{equation}
\end{lemma}

The following lemma will allow us to perturbate the diagonal subgroup $A=\{a(t): t \in \R\}$:
\begin{lemma}
\label{lemma:perturbate-diagonal-subgroup}
Let $H = (\SO(n,1))^k$, $A = \{a(t) = (b(\eta_1 t), b(\eta_2 t), \dots , b(\eta_k t)): t \in \R\}$ with $\eta_1 = \cdots = \eta_{k_1} = \eta > \eta_{k_1 +1}$, $\varphi: I \rightarrow (\R^{n-1})^k$ be an analytic curve as in Theorem \ref{thm:main-result} and $V$ be a representation of $H$. Suppose there exists a nonzero vector $v \in V$ such that 
   $$u_k(\varphi(s))v \in V^{-0}(A) \text{ for all } s \in I ,$$
then for all $s \in I$, $u_k(\varphi(s))v \in V^{-0}(A')$ where $A' = \{a'(t) = (b(\eta'_1 t), b(\eta'_2 t) \dots, b(\eta'_k t)): t \in \R\}$ with $\eta'_i = \eta_{k_1 +1}$ for $i = 1,2,\dots , k_1 +1$ and $\eta'_i = \eta_i$ for $i = k_1 +2, \dots, k$.
\end{lemma}
\begin{proof}
We first prove the following claim:
\begin{claim}
\label{claim:unipotent-invariance}
For all $s \in I$, $(u_k(\varphi(s))v)^{0}(A) $ is invariant under the following unipotent subgroup
 $$\{u_k(r (\varphi^{1})^{(1)}(s), \dots, r (\varphi^{k_1})^{(1)}(s), \vect{0}, \dots, \vect{0}): r\in \R\}.$$
\end{claim}
\begin{proof}[Proof of the claim]
On the one hand,
$$\begin{array}{rcl} a(t) u_k(\varphi(s + r e^{-\eta t}))v  & = & a(t) u_k(\varphi(s) + r e^{-\eta t}\varphi^{(1)}(s) + O(e^{-2\eta t}) )v \\ & = &  a(t) u_k(r e^{-\eta t}\varphi^{(1)}(s) + O(e^{-2\eta t}) ) a(-t) a(t) u_k(\varphi(s))v \\ & = & u_k(O(e^{-(\eta - \eta_{k_1 +1})t})) u_k(r (\varphi^{1})^{(1)}(s), \dots, r (\varphi^{k_1})^{(1)}(s), \vect{0}, \dots, \vect{0}) a(t)u_k(\varphi(s))v .\end{array}$$
Therefore as $t \rightarrow +\infty$, $$a(t)u_k(\varphi(s + r e^{-\eta t}))v \rightarrow u_k(r (\varphi^{1})^{(1)}(s), \dots, r (\varphi^{k_1})^{(1)}(s), \vect{0}, \dots, \vect{0}) (u_k(\varphi(s))v)^0(A).$$
On the other hand, 
$$a(t) u_k(\varphi(s + r e^{-\eta t})) v = (u_k(\varphi(s + r e^{-\eta t})) v)^0(A) + O(e^{-a t})$$
for some constant $a >0$. As $t \rightarrow +\infty$, $(u_k(\varphi(s + r e^{-\eta t})) v)^0(A) \rightarrow (u_k(\varphi(s))v)^0(A)$. Therefore, 
$$a(t) u_k(\varphi(s + r e^{-\eta t})) v \rightarrow (u_k(\varphi(s))v)^0(A) \text{ as } t \rightarrow +\infty.$$
\par Thus, we have that 
$$u_k(r (\varphi^{1})^{(1)}(s), \dots, r (\varphi^{k_1})^{(1)}(s), \vect{0}, \dots, \vect{0}) (u_k(\varphi(s))v)^0(A) = (u_k(\varphi(s))v)^0(A) \text{ for all } r \in \R .$$
This proves the claim.
\end{proof}
Let us define
 $$B_{k_1} := \{b_{k_1} (t) := (b(t), \dots , b(t), e, \dots, e): t \in \R \},$$
 where the first $k_1$ components are $b(t)$ and the rest are identity $e$. Now let us consider the $\SL(2, \R)$ copy in $H$ with $u_k(r (\varphi^{1})^{(1)}(s), \dots, r (\varphi^{k_1})^{(1)}(s), \vect{0}, \dots, \vect{0})$ corresponding to $\begin{bmatrix} 1 & r \\ 0 & 1\end{bmatrix}$, and $b_{k_1}(t)$ corresponding to $\begin{bmatrix}e^t & \\ & e^{-t}\end{bmatrix}$. Let us denote this subgroup by $\SL(2, k_1)$. Considering $V$ as a represenation of this $\SL(2,\R)$, we have that $(u_k(\varphi(s))v)^0(A)$ is invariant under $\left\{\begin{bmatrix}1 & r \\ 0 & 1\end{bmatrix}: r \in \R\right\}$. By basic theory of $\SL(2, \R)$ representations, we conclude that $(u_k(\varphi(s))v)^0(A) \in V^{+0}(B_{k_1})$. For $\epsilon >0$, define
  $$A_{\epsilon} := \{a_{\epsilon}(t): = a(t) b_{k_1}(-\epsilon t) : t \in \R\}.$$
 Since $(u_{k}(\varphi(s))v)^0(A) \in V^{+0}(B_{k_1})$, we have that $(u_{k}(\varphi(s))v)^0(A) \in V^{-0}(A_{\epsilon})$. Moreover, for $\epsilon >0$ small enough, $(u_{k}(\varphi(s))v)^-(A) \in V^- (A_{\epsilon})$. Therefore, for $\epsilon >0$ small enough,
 $$u_k(\varphi(s))v \in V^{-0}(A_{\epsilon}) \text{ for all } s \in I .$$
 By replacing $A$ by $A_{\epsilon}$ we can repeat the process above, untill we get $A_{\epsilon} = A'$.
 \par This completes the proof.
\end{proof}

Applying the above lemma several times until $\eta_1 = \eta_2 = \cdots = \eta_k$, we conclude the following:
\begin{corollary}
\label{cor:perturbate-diagonal-subgroup}
Suppose
\[\{u_k(\varphi(s))v\}_{s \in I} \subset V^{-0}(A),\]
then 
\[\{u_k(\varphi(s))v\}_{s \in I} \subset V^{-0}(A_{good}),\]
where $A_{good} : = \{a_{good} (t) := (b(t), b(t), \dots, b(t)) : t \in \R \}$.
\end{corollary}

The following result can be regarded as a corollary of Lemma \ref{lemma:sl2-representations}:
\begin{corollary}
\label{cor:sl2-representation}
Let $H$, $V$, $\varphi: I \rightarrow (\R^{n-1})^k$ and $A_{good} : = \{a_{good}(t): t\in \R\}$ be as above. Suppose a nonzero vector $v \in V$ satisfies that 
$$u_{k}(\varphi(s))v \in V^{-0}(A_{good}) \text{ for all } s\in I,$$
then $v$ is fixed by $H$.
\end{corollary}
\begin{proof}
Let us fix a point $s_0 \in I$ such that $(\varphi(s) - \varphi(s_0))\inv$ is defined for almost every $s \in I$ and $\{(\varphi(s_1)-\varphi(s_0))\inv - (\varphi(s_2)-\varphi(s_0))\inv: s_1, s_2 \in I\}$ is not contained in any proper subspace of $(\R^{n-1})^k$, then we have $u_{k}(\varphi(s_0))v \in V^{-0}(A_{good})$ and $u_{k}(\varphi(s) - \varphi(s_0)) u_{k}(\varphi(s_0))v \in V^{-0}(A_{good})$. Recall that $\SL(2, \varphi(s) - \varphi(s_0)) \subset H$ denotes the $\SL(2, \R)$ copy in $H$ such that $a_{good}(t)$ corresponds to $\begin{bmatrix}e^t & \\ & e^{-t}\end{bmatrix}$ and $u_{k}(\varphi(s) - \varphi(s_0))$ corresponds to $\begin{bmatrix}1 & 1 \\ 0 & 1\end{bmatrix}$ (see \S \ref{subsection:preliminaries}). Then by Lemma \ref{lemma:sl2-representations}, 
$$u_{k}(\varphi(s_0))v = u^{-}_{k}(-(\varphi(s) - \varphi(s_0))\inv) (u_{k}(\varphi(s_0))v)^0(A_{good})$$
for all $s \in I$ such that $(\varphi(s) -\varphi(s_0))\inv$ is defined. This implies that 
$(u_{k}(\varphi(s_0))v)^0(A_{good})$ is fixed by $u^{-}_{k}((\varphi(s_1) - \varphi(s_0))\inv-(\varphi(s_2) - \varphi(s_0))\inv)$ for any $s_1, s_2 \in I$ with the expression well defined. Since $\{(\varphi(s_1)-\varphi(s_0))\inv - (\varphi(s_2)-\varphi(s_0))\inv: s_1, s_2 \in I\}$ is not contained in any proper subspace of $(\R^{n-1})^k$, we have that $(u_{k}(\varphi(s_0))v)^0(A_{good})$ is fixed by the whole $U^-(A_{good})$. Note that $(u_{k}(\varphi(s_0))v)^0(A_{good})$ is also fixed by $A_{good}$. By basic facts of representation theory, we conclude that $(u_{k}(\varphi(s_0))v)^0(A_{good})$ is fixed by the whole $H$. Thus 
$v \in H (u_{k}(\varphi(s_0))v)^0(A_{good}) = (u_{k}(\varphi(s_0))v)^0(A_{good})$ is also fixed by $H$.
\par This completes the proof.
\end{proof}
Corollary \ref{cor:perturbate-diagonal-subgroup} and Corollary \ref{cor:sl2-representation} imply the following result:
\begin{lemma}
\label{lemma:basic-lemma-2}
Suppose there exists $v \in V \setminus{\vect{0}}$ such that 
\[\{u_k(\varphi(s))v\}_{s \in I} \subset V^{-0}(A),\]
then $v$ is fixed by $H$.
\end{lemma}
Finally we will finish the proof of Lemma \ref{lemma:basic-lemma-1}.
\begin{proof}[Proof of Lemma \ref{lemma:basic-lemma-1}]
Suppose $v \in V \setminus{\vect{0}}$ satisfies that $\{u_k(\varphi(s))v\}_{s\in I} \subset V^-(A)$, then by Lemma \ref{lemma:basic-lemma-2}, we have that $v$ is fixed by $H$, which contradicts the assumption.
\end{proof}

\section{Normal subgroups with closed orbits}
\label{section:normal-subgroups}

Recall that $a(t) = (b(\eta_1 t) , b(\eta_2 t) , \dots, b(\eta_k t))$ where 
\[\eta_{1} = \cdots = \eta_{k_1} > \eta_{k_1 +1} \geq \cdots \geq \eta_k.\]
Here we further assume that 
\[\begin{array}{l} \eta_{1} = \cdots = \eta_{k_1} \\
 > \eta_{k_1 +1} = \cdots = \eta_{k_2} \\ 
 > \eta_{k_2 + 1} =  \cdots \eta_{k_3} \\
  \cdots \\
  > \eta_{k_{j-1} + 1} = \cdots = \eta_{k_j} = \eta_k.
\end{array}\]
Recall that for $k' = 1,2,\dots, k$, $W_{k'} = \{ u_k(r \mathfrak{e}_{k'}): r \in \R\}$ is defined by \eqref{eq:W} and $W = W_{k_1}$. 
\par In this section, we will deal with the second difficulty discussed in \S \ref{subsection:outline-of-the-proof}. Let $\lambda_{\infty}$ be a limit measure of $\{\lambda_t: t >0\}$. Let us consider the following special case: there exists a normal subgroup $G_{1} \in \mathcal{L}$ such that
\[ \lambda_{\infty} (\pi(N(G_1, W))) >0, \text{ and } \lambda_{\infty}(\pi(S(G_1, W))) = 0. \]
\par For the case given as above, we will prove the following proposition:
\begin{proposition}
\label{prop:normal-subgroup}
$\lambda_{\infty}$ is invariant under the action of $W_{k_2}$.
\end{proposition}
\begin{proof}
\par By Theorem \ref{ratner}, we have that every $W$-ergodic component of the restriction of $\lambda_{\infty}$ on $\pi(N(G_1, W))$ is of form $g\mu_{G_1}$. Since $G_{1}$ is a normal subgroup of $G$, we have that $W \subset G_{1}$ and $N(G_1, W) = G$. This implies that every $W$-ergodic component of 
$\lambda_{\infty}$ is of form $g \mu_{G_1}$. In particular, $\lambda_{\infty}$ is $G_1$ invariant. 
\par Let $H_1 := H \cap G_1$, then $H_1$ is a normal subgroup of $H$ containing $W$. Note that $H = (\SO(n,1))^{k}$, we have that $H_1$ is of form $\prod_{i=1}^k Q_{i}$ where each $Q_i$ is either $\SO(n,1)$ or $\{e\}$. Let us call $Q_i$ the $i$th component of $H_1$. Since $W = W_{k_1} \subset H_1$,we have that $Q_{i} = \SO(n,1)$ for $i=1,\dots, k_1$. If the first $k_{2}$ components of $H_1$ are all $\SO(n,1)$, then there is nothing to prove ($W_{k_2}$ is contained in $H_1$ which fixes $\lambda_{\infty}$). Thus, later in the proof, we will assume that for some $ k_1 +1 \leq i  \leq k_{2}$, $Q_i = \{e\}$. Let $H(1) := H/H_1$. Since the first $k_1$ components of $H_1$ are all $\SO(n,1)$, $H(1)$ can be written as $\prod_{i= k_1 +1}^{ k} S_{i}$ where each $S_{i}$ is either $\SO(n,1)$ or $\{e\}$. For simplicity, we can realize $H(1)$ as $(\SO(n,1))^{k - k_1}$. For the same reason, we can realize $A(1) := A/ A \cap H_{1}$ as $\{a_1(t):=(b(\eta_{k_1 +1} t) , \dots, b(\eta_k t)) : t \in \R\}$ and $U^{+}(1) : = U^{+}(A)/ U^{+}(A)\cap H_1$ as 
\[\{ u'(\vect{x}_{k_1 + 1} , \dots, \vect{x}_k) := (u(\vect{x}_{k_1 + 1}) , \dots, u(\vect{x}_k)): \vect{x}_{k_1 + 1} , \dots, \vect{x}_k \in \R^{n-1}\}.\] 
Let $\varphi_{1}: I \rightarrow (\R^{n-1})^{k-k_1}$, and $z_{1}: I \rightarrow (\SO(n-1))^{k-k_1}$ respectively, denote the projection of $\varphi: I \rightarrow (\R^{n-1})^{k}$, and $z: I \rightarrow (\SO(n-1))^k$ respectively, onto the last $k - k_{1}$ components.
\par Let $G(1) := G / G_1$, $\Gamma(1) :=\Gamma / \Gamma \cap G_1$, then $X(1):= G(1)/\Gamma(1)$ is also a homogeneous space. The projection of $z(s)a(t)u_k(\varphi(s))x$ onto $X(1)$ can be realized as $z_1(s)a_1(t)u'(\varphi_1(s))x$. Let $\lambda_t(1)$ denote the normalized linear measure on $\{z_1(s)a_1(t)u'(\varphi_1(s))x: s \in I\}$, and $\lambda_{\infty}(1)$ denote the projection of $\lambda_{\infty}$ onto $X(1)$, then $\lambda_{\infty}(1)$ is a limit measure of $\{\lambda_t(1): t>0\}$, and $\lambda_{\infty}$ can be written as 
\begin{equation}
\label{equ:disintegrate-lambda}
\lambda_{\infty} = \int_{g \in \mathcal{F}(1)} g\mu_{G_1} \dd \lambda_{\infty}(1)(g),
\end{equation}
where $\mathcal{F}(1) \subset G(1)$ is a fundamental domain of $G(1)/\Gamma(1)$. By repeating the proof of Proposition \ref{prop:unipotent-invariance}, we conclude that $\lambda_{\infty}(1)$ is invariant under the action of
\[W(1):= \{u'(r \mathfrak{e}'_{k_2}) : r \in \R \},\]
where $\mathfrak{e}'_{k_2} \in (\R^{n-1})^{k-k_1}$ denotes the projection of $\mathfrak{e}_{k_2}$ onto the last $k- k_1$ components. Combined with \eqref{equ:disintegrate-lambda}, this implies that $\lambda_{\infty}$ is invariant under the action of $W_{k_2}$.
\end{proof}

\section{Conclusion}
\label{section:conclusion}
In this section, we will finish the proof of Theorem \ref{thm:main-result}. 
\begin{proof}[Proof of Theorem \ref{thm:main-result}]
For any convergent subsequence $\lambda_{t_i} \rightarrow \lambda_{\infty}$, we need to show that $\lambda_{\infty} = \mu_G$. Suppose not, then by Proposition \ref{prop:algebraic-condition}, there exist $L \in \mathcal{L}$ and $\gamma \in \Gamma$, such that 
\[\{u_k(\varphi(s))g\gamma p_L \}_{s \in I} \subset V^{-0}(A).\]
Applying Lemma \ref{lemma:basic-lemma-2}, we have that $g\gamma p_L$ is fixed by the action of $H$. Then $p_L$ is fixed by $(g\gamma)\inv H (g\gamma)$. Thus 
\[\begin{array}{rcl}
 \Gamma p_L & = & \overline{\Gamma p_L} \text{ (because $\Gamma p_L$ is discrete)} \\
            & = & \overline{\Gamma \gamma\inv g\inv H g \gamma p_L} \\
            & = & \overline{\Gamma g \inv H g \gamma p_L} \\
            & = & G p_L \text{ (because $\Gamma g \inv H$ is dense in $G$)}.
\end{array}\]
Let $G_{0}$ denote the connected component of $e$ in $G$, then $G_0 p_L = p_L$. 
Therefore $\gamma^{-1} g^{-1}H g
\gamma\subset G_0$ and $G_0 \subset N^1_{G}(L)$. By \cite[Theorem
2.3]{Shah_3}, there exists a closed subgroup $F_1 \subset
N^1_{G}(L)$ containing all $\mathrm{Ad}$-unipotent one-parameter
subgroups of $G$ contained in $N^1_{G}(L)$ such that $F_1 \cap
\Gamma$ is a lattice in $F_1$ and $\pi (F_1)$ is closed. If we put
$F= g\gamma F_1 \gamma^{-1} g^{-1}$, then $H\subset
F$ since $H$ is generated by it unipotent
one-parameter subgroups. Moreover, $Fx =g \gamma \pi(F_1)$ is closed
and admits a finite $F$-invariant measure. Then since
$\overline{Hx}=G/\Gamma$, we have $F=G$. This implies
$F_1 = G$ and thus $L$ is a normal subgroup of $G$. By Proposition \ref{prop:normal-subgroup}, $\lambda_{\infty}$ is invariant under the action of $W_{k_2}$ where $k_{2} > k_1$. By repeating the argument above with $W = W_{k_1}$ replaced by $W_{k_2}$, we conclude that there exists a normal subgroup $L_2 \in \mathcal{L}$ containing $W_{k_2}$ such that every $W_{k_2}$-ergodic component of $\lambda_{\infty}$ is of form $g\mu_{L_2}$. Repeating the proof of Proposition \ref{prop:normal-subgroup}, we can conclude that $\lambda_{\infty}$ is invariant under the action of $W_{k_3}$. Repeating this process, we will get that $\lambda_{\infty}$ is invariant under $W_{k_j} = W_{k}$. Moreover, there exists a normal subgroup $L_j \in \mathcal{L}$ containing $W_{k}$ such that every $W_k$-ergodic component of $\lambda_{\infty}$ is of form $g \mu_{L_j}$. Since $H \cap L_j$ is a normal subgroup of $H$ containing $W_k$, we have that $H\subset L_j$. Since $L_j$ is a
normal subgroup of $G$ and $\pi(L_j)$ is a closed orbit with finite
$L_j$-invariant measure, every orbit of $L_j$ on $G/\Gamma$ is also
closed and admits a finite $L_j$-invariant measure, in particular,
$L_j x$ is closed. But since $Hx$ is dense in
$G/\Gamma$, $L_j x$ is also dense. This shows that $L_j=G$, which
contradicts our hypothesis that the limit measure is not $\mu_G$.
\par This completes the proof. 
\end{proof}

\medskip

\bibliography{reference.bib}{}
\bibliographystyle{alpha}

\end{document}